\documentclass[A4paper,12pt]{article}
\usepackage{amsmath}
\usepackage{latexsym}
\usepackage{mathrsfs}
\usepackage{amssymb}
\usepackage{amscd}
\usepackage[dvips]{graphicx}                  

\newtheorem{theorem}{Theorem}
\newtheorem{corollary}{Corollary}

\newtheorem{proposition}{Proposition}

\newenvironment{proof}[1][Proof]{\textbf{#1.} }{\ \rule{0.5em}{0.5em}}
\long\def\symbolfootnote[#1]#2{\begingroup%
	\def\thefootnote{$\;$}\footnote[#1]{$^*$#2}\endgroup}
\begin{document}
	
	\title{On special partitions of metric spaces}
	\author{Ryszard Frankiewicz and Joanna Jureczko\footnote{The second author is partially supported by Wroc\l{}aw Univercity of Science and Technology grant of K34W04D03 no. 8201003902.}}
\maketitle

\symbolfootnote[2]{Mathematics Subject Classification: Primary 03C25, 03E35, 03E55, 54E52.

	\hspace{0.2cm}
	}

\begin{abstract}
The main result of this paper is to show that, if $\kappa$ is the smallest real-valued measurable cardinal not greater than $ 2^{\aleph_0}$, then there exists a complete metric space of cardinality not greater than $ 2^{\kappa}$ admitting a Kuratowski partition.   
\end{abstract}

\maketitle

\section{Introduction}

In  1935, K. Kuratowski, in \cite{KK},  followed the results of Lusin published in 1912, (\cite{NL}),  posed the problem  whether a function $f \colon X \to~Y$, (where $X$ is completely metrizable and $Y$ is metrizable), such that each preimage of an open set of $Y$ has the Baire property, is continuous apart from a meager set.

In the 1970s, R. H.  Solovay (unpublished result) proved, using forcing methods, (precisely the generic ultrapower), that any partition of $[0,1]$ into Lebesgue-measure zero sets produces a non-measurable set. A few years later, L.  Bukovsky \cite{LB} advanced a shorter and less complicated proof than that of Solovay.

At about the same time, A. Emeryk, R. Frankiewicz, and W. Kulpa  \cite{EFK} demonstrated that Kuratowski's problem is equivalent to asserting the existence of partitions of completely metrizable spaces into meager sets such that the union of each subfamily  of this partition has the Baire property. Such a partition is called a \textit{Kuratowski partition}, or, stated more briefly, a \textit{$K$-partition}. (The next section provides a formal definition.)

Our paper provides a perspective to the problem other than that of R. Frankiewicz and K. Kunen \cite{FK}, who, among others, proved that, if ZFC + "there exists a $K$-partition of a Baire metric space" is consistent, then ZFC + "there exists measurable cardinal" is consistent as well, by using forcing methods and the localization property. 

Thus, the main result of this paper is to show that, if $\kappa$ is the smallest real-valued measurable cardinal not greater than $ 2^{\aleph_0}$, then there exists a complete metric space of cardinality not greater than $ 2^{\kappa}$ admitting Kuratowski partition. 

Previously, we thought to use for this purpose a special ideal type associated with $K$-partition of a given space, which \cite{JJ} refers to as a \textit{$K$-ideal}, but this was unworkable for two reasons. First, we incorrectly assumed that, based on the structure of such a $K$-ideal, complete information about the $K$-partition of a given space could be “decoded”. Unfortunately, this is not the case, however, because, as shown in \cite{JJ}, the structure of such an ideal can be almost arbitrary, i.e., it can be the Fr\'echet ideal, and so, by \cite[Lemma 22.20, p. 425]{TJ}, it is not precipitous when $\kappa$ is regular. (To show this, we use, as in \cite{JJ}, direct sums of a space.)
Moreover, as demonstrated in \cite{JJ}, for a measurable cardinal $\kappa$, a $\kappa$-complete ideal can be represented by some $K$-ideal. However, if $\kappa =|\mathcal{F}|$ is not a measurable cardinal, where $\mathcal{F}$ is a $K$-partition of a given space, one can obtain an $|\mathcal{F}|$-complete ideal that can be the Fr\'echet ideal or a $\kappa$-complete ideal representing some $K$-ideal, or it can be a proper ideal of such a $K$-ideal and so contains the Fr\'echet ideal. Thus, for obtaining a $K$-partition from a $K$-ideal, we need to have complete information about the space in which the ideal is being considered.

Secondly, if we want to use the "idea" of a K-ideal, we must additionally assume that the space $X(I)$, where $I$ is K-ideal, is complete (see \cite{FK} or the next section for formal definitions) because, as will be shown in Proposition 1 and Theorem 1 (see Section 3), only if $X(I)$ is complete can $I$ be maximal. Thus, the assumption used in Theorem 2 is the only one with which we can show the existence of a complete metric space with $K$-partition.

From \cite{EFK}, it follows that the restriction "$\kappa \leqslant 2^{\aleph_0}$" on the cardinality of the $K$-partition used in Theorem 2 comprises the real restriction. (To simplify, we assume that $\kappa$ is real-measurable.) Moreover, a space with smaller cardinality admits the existence of a $K$-partition, but this does not enlarge on the completion of this space. (Compare with \cite{JJ}.) 

In summary, Theorem 2 (which also applies to an arbitrary metric space) was a missing result in $K$-partition considerations and so serves to complete investigations of Kuratowski's problem.

As a consequence of our main result (i.e., Theorem 2), we were able to obtain the result presented in \cite{FK}. (See Fact 2 in Section 3 and \cite[Lemma 5.1]{FJW}.)

Worth mentioning is that the subject of this paper has wider applications, including some in measurable selector theory and related topics.

This paper consists of three  sections. Section 2 provides relevant definitions and previous results, including those involving $K$-partitions, precipitous ideals, and measurable cardinals. Our main results, including Theorem 2, are presented in Section 3, and the paper’s final section then discusses Theorem 2’s implications.

Section 2 provides a number of well-known definitions. For definitions and facts not cited there, however, we refer interested readers to, e.g., \cite{KK1} (topology) and \cite{TJ} (set theory).

\section{Definitions and previous results}

Throughout the paper, we assume that $X$ is a Baire space, i.e., a space in which the Baire theorem holds.
\\\\
\textbf{2.1.}   A set $U \subseteq X$ has \textit{the Baire property} iff there exists an open set $V \subset X$ and a meager set $M \subset X$ such that $U = V \triangle M$, where $\triangle$ represents the symmetric difference of sets.
\\
\\
\textbf{2.2.} A partition $\mathcal{F}$ of $X$ into meager subsets of $X$ is called a \textit{Kuratowski partition}, (or, in short, a \textit{$K$-partition}), iff $\bigcup \mathcal{F}'$ has the Baire property for all $\mathcal{F}' \subseteq \mathcal{F}$.

Throughout this paper, we assume that $$\kappa = \min\{|\mathcal{F}|\colon \mathcal{F} \textrm{ is $K$-partition of }X \}.$$
For a given regular cardinal $\kappa$, we enumerate a $K$-partition as follows:
 $$\mathcal{F} = \{F_\alpha \colon \alpha < \kappa\}.$$

(We can assume that each subspace of $X$ with cardinality less than $\kappa$ is not Baire.)

For an open set $U \subseteq X$ treated as a subspace of $X$ that is Baire, the family $$\mathcal{F}\cap U = \{F \cap U \colon F \in \mathcal{F}\}$$ is a $K$-partition of $U$.
\\
\\
\textbf{2.3.} With any $K$-partition
$\mathcal{F} = \{F_\alpha \colon \alpha < \kappa\}$ indexed by a regular cardinal $\kappa$,  one may associate an ideal 
$$I_\mathcal{F} = \{A \subset \kappa \colon \bigcup_{\alpha \in A} F_\alpha \textrm{ is meager}\},$$
which is called the \textit{$K$-ideal} (see \cite{JJ}).
\\
Note that $I_\mathcal{F}$ is a non-principal ideal. Moreover, $[\kappa]^{< \kappa} \subseteq I_{\mathcal{F}}$ because $\kappa = \min\{|\mathcal{G}| \colon \mathcal{G} \textrm{ is a $K$-partition of }X\}$.
\\
\\
\textbf{2.4.} Let $I$ be an ideal on $\kappa$ and let $S$ be a set with positive measure, i.e., $S \in P(\kappa) \setminus I$. (For convenience, we use $I^+$ instead of $P(\kappa) \setminus I$ throughout). 
\\An \textit{$I$-partition} of $S$ is a maximal family $W$ of subsets of $S$ of positive measure such that $A \cap B \in I$ for all distinct $A, B \in W$.

An $I$-partition $W_1$ of $S$ is a \textit{refinement} of an $I$-partition $W_2$ of $S$ ($W_1 \leq~W_2$) iff each $A \in W_1$ is a subset of some $B\in W_2$.

If $I$ is a $\kappa$-complete ideal on $\kappa$ containing singletons, then $I$ is \textit{precipitous} iff, whenever $S \in I^+$ and $\{W_n \colon n < \omega\}$ is a sequence of  $I$-partitions of $S$ such that 
$W_0 \geq W_1\geq ... \geq W_n \geq ...$,
there exists a sequence of sets
$X_0 \supseteq X_1\supseteq ... \supseteq X_n \supseteq ...$
such that $X_n \in W_n$ for each $n\in \omega$ and $\bigcap_{n=0}^{\infty} X_n \not = \emptyset$ (see also \cite[p. 424-425]{TJ}).

Let $X$ be a Baire metric space and $\mathcal{F}$ be a $K$-partition of $X$.
The ideal $I_{\mathcal{F}}$ is \textit{everywhere precipitous} iff $I_{\mathcal{F}\cap U}$ is precipitous for each non-empty open set $U \subseteq X$. (Obviously, it can be deduced from Union Theorem  \cite[p. 82]{KK1}). 
	\\\\
	\textbf{Fact 1 (\cite{FJ})} Let $X$ be a Baire metric space with $K$-partition $\mathcal{F}$ of cardinality $\kappa = \min\{|\mathcal{G}| \colon \mathcal{G} \textrm{ is a $K$-partition for  }X\}$.   Then there exists an open set $U \subset X$ such that the $K$-ideal $I_{\mathcal{F}\cap U}$ on $\kappa$ associated with $\mathcal{F}\cap U$ is precipitous.	
\\
\\	
\textbf{2.5.}
Let $\lambda$ be a cardinal.
An ideal $I$ is \textit{$\lambda$-saturated} iff there exists no $I$-partition $W$ of size $\lambda$.
Then,
$$sat(I) \textrm{ is the smallest $\lambda$ such that $I$ is $\lambda$-saturated}.$$

\noindent
\textbf{2.6.}
An uncountable cardinal $\kappa$ is \textit{measurable} iff there exists a non-principal maximal and $\kappa$-complete ideal on $\kappa$.
\\
\\
\textbf{Fact 2 (\cite{FK}) }  ZFC +  "there exists measurable cardinal" is equiconsistent with ZFC + "there exists a Baire metric space with a $K$-partition of cardinality $\kappa$".
\\
\\
\textbf{Fact 3 (\cite{TJ} )} 
(a) If $\kappa$ is a regular uncountable cardinal that carries a precipitous ideal, then $\kappa$ is measurable in some transitive model of ZFC.

(b) If $\kappa$ is a  measurable cardinal, then there exists a generic extension in which $\kappa = \aleph_1$, and $\kappa$ carries a precipitous ideal.
\\
\\
\textbf{2.7.} Let $I$ be an ideal over a cardinal $\kappa$, and let 
$$X(I) = \{x \in (I^+)^\omega \colon \bigcap \{x(n) \colon n \in \omega\} \not = \emptyset \textrm{ and } \forall_{n \in \omega}\ \bigcap\{x(m) \colon m < n\} \in I^+\}.$$
The  set $X(I)$ is considered to be a subset of a complete metric space $(I^+)^\omega$, where the set $I^+$ is equipped with the discrete topology (see also \cite{FK}).
\\\\
\textbf{Fact 4 (\cite{FK}) } $X(I)$ is a Baire space iff $I$ is a precipitous ideal.
\\\\
\textbf{Fact 5 (\cite{FK})} Let $I$ be a precipitous ideal over some regular cardinal. Then there is a $K$-partition of $X(I)$.
\\\\
\textbf{2.8.} 
A \textit{nontrivial measure} on  $X$ is a map $\mu\colon P(X) \to [0,1]$  such that $\mu$ is a countably additive measure vanishing on points with $\mu(X) =1$ (where $P(X)$ represents the power set of $X$).

A measure  $\mu$ is \textit{$\kappa$-additive} whenever $\{A_\xi\colon \xi < \lambda\}$ is a family of sets of measure zero and $\lambda< \kappa$ then $\bigcup_{\xi < \lambda} A_\xi$ is measure zero.
There exists the largest $\kappa$ such that $\mu$ is $\kappa$-additive. Then,
$$add(\mu) = \min \{\kappa \colon \mu (\bigcup_{\xi < \kappa}A_\xi) > 0, \mu (A_\xi) = 0\}.$$

A cardinal $\kappa$ is \textit{real-valued} iff $\kappa$ carries a nontrivial $\kappa$-additive measure.
\\\\
\textbf{Fact 6 (\cite{SU, RS})} Let $\kappa$ be real-valued measurable. If $\kappa \leqslant 2^{\aleph_0}$, then there is an extension $\mu$ of Lebesgue measure defined on all subsets of $\mathbb{R}$ with $add(\mu) = \kappa$.
\\\\
\textbf{Fact 7 (\cite{RS})}  The following theories are equiconsistent.
\\(1) ZFC + "there is a measurable cardinal".
\\(2) ZFC + "a Lebesgue measure has a countably additive extension $\mu$ defined on every set of reals".
\\\\
\textbf{Fact 8 (\cite{RS})} Let $\kappa$ be a real-valued measurable cardinal, and let $\mu$ be a nontrivial $\kappa$-additive real-valued measure on $\kappa$. Then, $I = \{A\subseteq \kappa \colon \mu(A) = 0\}$ is a nontrivial ideal in $P(\kappa)$.
\\\\
\textbf{Fact 9 (Ulam, \cite{RS})}  Let $\kappa$ be a real-valued measurable cardinal, and let $\mu$ be a nontrivial  measure on $\kappa$. Then, $I = \{A\subseteq \kappa \colon \mu(A) = 0\}$ is $\aleph_1$-saturated.
\\\\
Obviously, $I$ defined in Fact 8 and Fact 9 is precipitous (compare with \cite[Lemma 22.22]{TJ}).

\section{Main results}

\begin{proposition} Let $X$ be a space with K-partition $\mathcal{F}$. Let $\kappa$ be a regular cardinal such that $\kappa = \{|\mathcal{G}| \colon \mathcal{G} \textrm{ is a K-partition of } X\}$. Let $I_{\mathcal{F}}$ be a K-ideal associated with $\mathcal{F}$. If $X(I_{\mathcal{F}})$ is complete, then there exists an open set $U \subset X$ such that  $I_{\mathcal{F}\cap U}$ is maximal.
	\end{proposition}

\begin{proof}
By assumption, $X(I_{\mathcal{F}})$ is complete, and hence the Baire theorem holds. By Fact 4, $I_{\mathcal{F}}$ is precipitous. Without the loss of generality we can assume that $I_{\mathcal{F}}$ is everywhere precipitous.
Let $\mathcal{F} = \{F_\alpha \colon \alpha < \kappa\}$.

Suppose that $I_{\mathcal{F}\cap V}$ is not maximal for any open  $V\subset X$. Without loss of generality, we can assume that $sat(I_{\mathcal{F}\cap V})$ is infinite for any open  $V\subset X$. Let $\mathcal{U}$ be a disjoint family of open sets, such that $\bigcup \mathcal{U}$ is dense in $X$.  Fix $U \subset \mathcal{U}.$ Hence, there exists a countable family $$\mathcal{A}^U = \{A_n^U \colon n \in \omega, A_n^U\in I^+_{\mathcal{F}\cap U}, A_n^U \cap A_m^U = \emptyset \textrm{ for any } n \not = m\}.$$
Define $$B_0^U = \bigcup_{n=0}^{\infty} A_n^U, B_1^U = \bigcup_{n=1}^{\infty} A_n^U, ..., B_k^U = \bigcup_{n=k}^{\infty} A_n^U, ...\ . $$ Then,
$\bigcap_{k =0}^m B_k^U \in I^+_{\mathcal{F}\cap U}$ for all $m<n, n \in \omega$, (i.e., each finite intersection $\bigcap_{k =0}^m B_k^U$ belongs to $I^+_{\mathcal{F}\cap U}$), but $\bigcap_{k \in \omega} B_k^U = \emptyset$. 

For any $ k\in \omega$ define $B_k = \bigcup_{U \in \mathcal{U}} B_k^U$. Obviously, $B_k \in I^+_{\mathcal{F}}$. Indeed. The set $\bigcup_{\alpha \in B^U_k} F_\alpha$ is non-meager. Then, by Union Theorem, \cite[p. 82]{KK1}, the set $\bigcup_{U \in \mathcal{U}}\bigcup_{\alpha \in B^U_k} F_\alpha$ is also non-meager and $\bigcup \mathcal{U}$ is dense in $X$.

Now, choose a sequence $(x_m)_{m \in \omega}$ of elements of $X(I_\mathcal{F})$ having the following properties:
\begin{itemize}
	\item [(1)] $x_m = (x_m(k))_{k \in \omega}$, where $x_m(k) = B_k$ for $k \leqslant m$ and arbitrary element of $I^+_\mathcal{F}$ for $k>m$ but such that $x_m \in X(I_\mathcal{F})$. 
	\item [(2)] $d(x_m, x_n) \leqslant \frac{1}{2^{n-m}}$, for all $m\leqslant n, m, n \in \omega$ (where $d$ means  the metric in $X(I_\mathcal{F})$). 
\end{itemize}
Since $\bigcap_{k =0}^m B_k \in I^+_\mathcal{F}, m \in \omega$, the sequence $(x_m)_{m \in \omega}$ fulfills the Cauchy condition, but the limit of  $(x_m)_{m \in \omega}$ does not belong to $X(I_\mathcal{F})$, because $\bigcap_{k \in \omega} B_k = \emptyset$, thus contradicting the assumption of completeness of $X(I_\mathcal{F})$.  Hence $I_{\mathcal{F}\cap U}$ is maximal for some $U \subset X$.
\end{proof}

\begin{theorem}
	Let $X$ be a complete metric space with $K$-partition $\mathcal{F}$ of cardinality $\kappa$, where $\kappa = \min \{|\mathcal{G}| \colon \mathcal{G} \textrm{ is a $K$-partition of } X\}$ is regular, and let $I_{\mathcal{F}}$ be a $K$-ideal associated with $\mathcal{F}$. If $X(I_{\mathcal{F}})$ is complete, then $\kappa$ is measurable. 
\end{theorem}

\begin{proof}
		Let $\mathcal{F}$ be a $K$-partition of $X$ of cardinality $\kappa$. Let $I_{\mathcal{F}}$ be a $K$--ideal associated with $\mathcal{F}$. By Fact 1, there exists a non-empty open set $U \subseteq X$ such that $I_{\mathcal{F}\cap U}$ is a precipitous ideal. Without loss of generality, we can assume that $I_{\mathcal{F}\cap U}$ is everywhere precipitous and hence is $\kappa$-complete. By the remark given in Section 2.3, $I_{\mathcal{F}\cap U}$ is non-principal. By Proposition 1, $I_{\mathcal{F}\cap U}$ is maximal. Hence, $\kappa$ is measurable.
\end{proof}

\begin{theorem}
	Let $\kappa$ be a regular  and  the smallest real-valued measurable cardinal such that $\aleph_1 < \kappa \leqslant 2^{\aleph_0}$. Then, there exists a complete metric  space of cardinality not greater than $2^{\kappa}$ which admits $K$-partition.
\end{theorem}

\begin{proof} For simplification we can assume that $X=[0,1]$. \\
	Let $\mu\colon P([0,1]) \to [0,1]$ be a nontrivial $\kappa$-additive measure. Then, by Fact 6, we can assume that $\mu$ extends everywhere a Lebesgue measure  on $[0,1]$. 
	Let $A, B \in P([0,1])$ be $\mu$-measurable sets. Define a relation 
	$$A \sim B \textrm{ iff } \mu(A\triangle B) = 0,$$
	where $\triangle$ indicates  the symmetric difference of sets. 
	Note that $\sim$ above is the equivalence relation. If $A \in P([0,1])$ is  $\mu$-measurable, then $[A]$  denotes the equivalence class determined by  $A$.
	Define a metric
	$$\rho([A], [B]) = \mu(A\triangle B).$$ 
	Since $A, B \in P([0,1])$ are $\mu$-measurable, $\rho$ is well defined.
	
Define
	$$Y = \{[A] \colon A \in P([0,1]), A \textrm{ is } \mu-\textrm{measurable}\}.$$
	The space $(Y, \rho)$ is complete. Indeed.  Let $([A_n])_{n \in \omega}$ be a sequence fulfilling the Cauchy condition. Then, $[\bigcap_{n \in \omega}\bigcup_{k \in \omega}A_{n+k}]$ is its limit point.
	
	Enumerate the elements of the interval $[0,1]$ by $\{x_\alpha \colon \alpha < \mathfrak{c}\}$.
	Let $U_{x_\alpha}$ denotes a neighbourhood of $x_\alpha$, $\alpha < \mathfrak{c}$, ($U_{x_\alpha} \subset [0,1]$ is an "open" generator such that $\mu(U_{x_\alpha})>0$).
	
	 For any $\alpha < \mathfrak{c}$, define 
	$$F_\alpha =\{[A]\in Y \colon \alpha = \min \{\beta < \mathfrak{c} \colon \forall_{U_{x_\beta}}\  \mu(U_{x_\beta} \cap A) >0\}\}.$$
	Obviously $F_\alpha \cap F_\beta = \emptyset$ for any $\alpha, \beta < \mathfrak{c}$, $\alpha \not = \beta$ and $F_\alpha$ is meager in $Y$, for any $\alpha < \mathfrak{c}$, (because $Y$ as a complete metric space, hence fulfills the Baire Theorem).

	Now, let $B$ be a subset of indices of the family $\{F_\alpha \colon \alpha < \mathfrak{c}\}$ and let  $\mu(B)>0$, then $\bigcup_{\alpha \in B}F_\alpha$ has the Baire property because it contains $$V(B)=\{[A] \in Y \colon \exists_{\alpha \in B}\ \exists_{U_{x_\alpha}}\  \mu(U_{x_\alpha} \cap A)>0\},$$ which is open and dense. Indeed.  For arbitrary $[A] \in V(B)$, we have $$\forall_{U_{x_\beta}}\ \mu(U_\beta \cap A) > 0 \textrm{ implies } \exists_{U_{x_\beta}}\  \mu(U_{x_\beta} \cap A)>0$$ and   $$\min \{\beta < \mathfrak{c} \colon \forall_{U_{x_\beta}}\ \mu(U_{x_\beta}\cap A)>0\} \in B.$$ 
	
	In the case where $\mu(B)=0$, there exists an open and dense $G_\delta$ - set $$V(B')=\{[A]\in Y \colon \mu(U_{x_\alpha} \cap A)>0 \textrm{ for some } U_{x_\alpha} \textrm{ and } \alpha \not \in B\}$$ which is contained in $\bigcup_{\alpha \not \in B} F_\alpha $.
	Hence, $\{F_\alpha \colon \alpha < \mathfrak{c}\}$ is a $K$-partition of $Y$.
\end{proof}
\\

Note that the above result is true in ZFC only, and the given space has to have density greater than $2^{\aleph_0}$ (see \cite{EFK}).

\section{Consequences}

In this section, we present the consequences of Theorem 2. The interested reader can find still more consequences in \cite{JJ1}.
\\

A map $f \colon X \to Y$ has \textit{the Baire property} iff for each open set $V \subset Y$, $f^{-1}(V)$ has the Baire property.
\\\\
\textbf{Fact 11 (\cite{FJW})}  Let $X, Y$ be topological spaces and $A \subset X$. The following statements are then equivalent:
	\\
	(a) The set $A$ does not admit Kuratowski partition.
	\\
	(b) For any mapping $f \colon A \to Y$ having the Baire property, there exists a meager set $M \subset A$ such that $f\upharpoonright(A\setminus M)$ is continuous.
\\

The immediate corollary following from Fact 2, Fact 11, and Theorem 2 is as follows:

\begin{corollary}
The following theories are consistent:
\\
(1) ZFC + "there is a measurable cardinal",
\\
(2) ZFC + "there is a complete metric space $X$ of cardinality not greater than $2^\mathfrak{c}$ and a function $f \colon X \to Y$ having the Baire property such that there is no meager set $M \subseteq X$ for which $f\upharpoonright(X \setminus M)$ is continuous".
\end{corollary}

{\sc Ryszard Frankiewicz}
\\
Silesian University of Technology, Gliwice, Poland
\\
{\sl e-mail: ryszard.frankiewicz@polsl.pl}
\\

{\sc Joanna Jureczko}
\\
Wroc\l{}aw University of Science and Technology, Wroc\l{}aw, Poland
\\
{\sl e-mail: joanna.jureczko@pwr.edu.pl}

\end{document}